\documentclass[a4paper,12pt,reqno]{amsart}
\usepackage{amssymb,amsxtra,paralist,mathrsfs,upgreek,xcolor}
\usepackage[colorlinks=true,linkcolor=magenta,citecolor=cyan,urlcolor=olive]{hyperref}
\usepackage[normalem]{ulem}

\title[Growth Estimates for Solutions to the Wave Equation on Damek--Ricci Spaces]{Growth Estimates for Solutions to \\ the Wave Equation on Damek--Ricci Spaces}

\author{Yunxiang Wang, Lixin Yan, and Hong-Wei Zhang}
\address{Yunxiang Wang, Department of Mathematics, Sun Yat-sen University, Guangzhou, 510275, P. R. China}
\email{\href{mailto: wangyx677@mail.sysu.edu.cn}{wangyx677@mail.sysu.edu.cn}}
\address{Lixin Yan, Department of Mathematics, Sun Yat-sen University, Guangzhou, 510275, P. R. China}
\email{\href{mailto: mcsylx@mail.sysu.edu.cn}{mcsylx@mail.sysu.edu.cn}}
\address{Hong-Wei Zhang, 
Institute for Advanced Study in Mathematics, Harbin Institute of Technology, Harbin, 150001, P. R. China; Institut für Mathematik, Universität Paderborn, Paderborn, 33098, Germany}
\email{\href{mailto: zhanghw@hit.edu.cn}
{zhanghw@hit.edu.cn}}

\date{}


\numberwithin{equation}{section}
\newtheorem{theorem}{Theorem}
\newtheorem{lemma}[theorem]{Lemma}

\newtheorem{proposition}[theorem]{Proposition}
\theoremstyle{definition}

\theoremstyle{remark}

\numberwithin{theorem}{section}

\newcommand{\R}{\mathbb{R}}
\newcommand{\BB}{\mathbb{B}}
\newcommand{\e}{\mathrm{e}}
\newcommand{\Id}{\mathrm{Id}}

\newcommand{\supp}{\mathrm{supp}}

\newcommand{\bc}{\mathbf{c}}
\newcommand{\F}{\mathscr{F}}
\newcommand{\g}{\mathbf{g}}
\newcommand{\fa}{\mathfrak{a}}
\newcommand{\fn}{\mathfrak{n}}
\newcommand{\fv}{\mathfrak{v}}
\newcommand{\fz}{\mathfrak{z}}
\newcommand{\fs}{\mathfrak{s}}
\newcommand{\fh}{\mathfrak{h}}
\newcommand{\cS}{\mathcal{S}}

\renewcommand{\Re}{\mathrm{Re}}
\renewcommand{\Im}{\mathrm{Im}}
\renewcommand{\SS}{\mathbb{S}}
\renewcommand{\i}{\mathrm{i}}
\renewcommand{\d}{\mathrm{d}}

\renewcommand{\L}{\mathcal{L}}
\newcommand{\red}{\color{red}}

\usepackage{CJKutf8}

\allowdisplaybreaks

\begin{document}

\begin{abstract}
Let $\mathcal{L}$ be the positive definite left-invariant distinguished Laplacian, and let $\mathrm{d}\rho$ denote the right Haar measure on a Damek--Ricci space $S$. Let $u(t,x)$ denote the solution to the wave equation $\partial_t^2 u + \mathcal{L} u=0$ with initial data $(u,\partial_t u)|_{t=0}=(f,g)$. In this paper, we establish the sharp-in-regularity $L^p$-bounds
\begin{align*}
\|u(t,\cdot)\|_{L^p(S ,\mathrm{d}\rho)}
\lesssim_p(1+|t|)^{2|\frac{1}{p}-\frac{1}{2}|}\|(\mathrm{Id}+\mathcal{L})^{\frac{\alpha_0}{2}}\!f\|_{L^p(S ,\mathrm{d}\rho)}+(1+|t|)\,\|(\mathrm{Id}+\mathcal{L})^{\frac{\alpha_1}{2}}\!g\|_{L^p(S,\mathrm{d}\rho)}
\end{align*}
for all $t\in\mathbb{R}^*$ and $1<p<\infty$, where the exponents $\alpha_0 = (n-1)\left|1/p-1/2\right|$ and $\alpha_1 = (n-1)\left|1/p-1/2\right| -1$ attain their critical values. This result settles, in full generality, the conjecture raised by Müller, Thiele, and Vallarino.
\end{abstract}

\maketitle


\section{Introduction}\label{sec1}

The wave equation has long played a central role in analysis and partial differential equations on manifolds, where the non-flat geometry strongly influences wave propagation. In particular, on negatively curved manifolds, the long-time dispersive estimates for the wave equation differ significantly from the classical results in the Euclidean setting. This phenomenon has been observed in various contexts, for instance, in \cite{AP14,AnPiVa,MT11,Tat01} on real hyperbolic spaces, in \cite{AZ24,Zha21} on their higher-rank generalizations, that is, on noncompact symmetric spaces, in \cite{SSWZ20} on non-trapping asymptotically hyperbolic manifolds, and in \cite{APV15} on the so-called Damek--Ricci spaces.

Damek–Ricci spaces are solvable extensions $S = N \rtimes \mathbb{R}^+$ of Heisenberg-type groups $N$, endowed with a left-invariant Riemannian structure. As Riemannian manifolds, these solvable Lie groups include important examples such as real hyperbolic spaces and all other noncompact rank-one symmetric spaces. Most of them, however, are not symmetric, thereby providing counterexamples to the Lichnerowicz conjecture \cite{DaRi1}. We refer to Sect. \ref{sec2} for more details on the structure of $S$ and the analysis carried out on it.

The dispersive effect of the wave equation, manifested through the $L^{p'}\to L^p$-estimates of the propagator, describes how the solution spreads out over time. In contrast, the $L^p$-norm ($p\neq2$) estimates for solutions to the wave equation may not exhibit decay but instead show growth in time. The following Sobolev estimate was first established by Müller and Thiele \cite{MuTh} on the group model of the real hyperbolic space, namely the $ax+b$ group, which corresponds to the case where the nilpotent part $N$ is abelian and thus represents the simplest instance of a Damek–Ricci space. It was later extended to the full class of Damek–Ricci spaces by Müller and Vallarino \cite{MuVa}.

\begin{theorem}[Müller \& Vallarino \cite{MuVa}]\label{thm:MullerVallarino}
Let $S$ be a Damek–Ricci space of dimension $n$, and let $\L$ denote the positive definite left-invariant distinguished Laplacian and $\d\rho$ the right Haar measure on $S$. Then, for all $1\leq p  \leq \infty$, there exists $C_p>0$ such that the solution to the Cauchy problem
\begin{align}\label{eq: wave}
\begin{cases}
\partial_t^2 u(t,x)+\L u(t,x)=0,
\qquad t\in\mathbb{R},\,x\in S,  \\
u(0,x)=f(x),\,\partial_t u(t,x)|_{t=0}=g(x),
\end{cases}
\end{align}
satisfies the estimate 
\begin{align}\label{est:mainthm}
\|u(t,\cdot)\|_{L^p(S,\d\rho)}
\leq C_p\,
\Big((1+|t|)^{2|\frac{1}{p}-\frac{1}{2}|}\|(\Id+\L)^{\frac{\alpha_0}{2}}\!f\|_{L^p(S,\d\rho)}+(1+|t|)\,\|(\Id+\L)^{\frac{\alpha_1}{2}}\!g\|_{L^p(S,\d\rho)}\Big)
\end{align}
for all $t\in\mathbb{R}^*$, provided that 
\begin{align*}
\alpha_0>(n-1)\left|{1\over p}-{1\over 2}\right|
\quad\text{and}\quad \alpha_1>(n-1)\left|{1\over p}-{1\over 2}\right|-1.
\end{align*}
\end{theorem}

Müller and Thiele on the $ax+b$ group \cite[Remark 8.1]{MuTh}, and Müller and Vallarino on Damek–Ricci spaces \cite[Remark 2]{MuVa}, conjectured that estimate \eqref{est:mainthm} should remain valid at the endpoint cases $\alpha_0=(n-1)\left|1/p-1/2\right|$ and $\alpha_1=(n-1)\left|1/p-1/2\right|-1$. This conjecture was recently confirmed on the $ax+b$ group by the first- and second-named authors in \cite{WaYa}. The following result extends this conclusion to the entire class of Damek–Ricci spaces.

\begin{theorem}\label{thm:main}
Theorem \ref{thm:MullerVallarino} remains valid at the endpoint cases 
\begin{align*}
\alpha_0=(n-1)\left|{1\over p}-{1\over 2}\right|
\quad\text{and}\quad \alpha_1=(n-1)\left|{1\over p}-{1\over 2}\right|-1,
\end{align*}
and these regularity conditions are sharp.
\end{theorem}

Let us briefly comment on Theorem \ref{thm:main}. This endpoint result represents a strengthening of Theorem \ref{thm:MullerVallarino}. In fact, we now know that estimate \eqref{est:mainthm} holds if and only if $\alpha_0 \ge (n-1)\left|1/p-1/2\right|$ and $\alpha_1 \ge (n-1)\left|1/p-1/2\right|-1$, which corresponds to the analogue of the related results in the Euclidean setting, see \cite{Mi,Pe}.

However, as pointed out by Müller and Thiele \cite[p. 147]{MuTh}, the exponent $2|1/p - 1/2|$ of $(1 + |t|)$ in estimate \eqref{est:mainthm} is independent of the topological dimension $n$. This contrasts with the corresponding estimates of Miyachi \cite{Mi} and Peral \cite{Pe} in the Euclidean case, where the exponent is $(n - 1)|1/p - 1/2|$. We observe here that the corresponding exponent could be $(\nu - 1)|1/p - 1/2|$, where $\nu$ denotes the pseudo-dimension (or dimension at infinity), and $\nu=3$ when $S$ is a Damek–Ricci space. This natural conjecture is motivated by previously known results on dispersive estimates in similar settings, such as real hyperbolic spaces, Damek–Ricci spaces, and general noncompact symmetric spaces, where the long-time dispersive decay rate depends on the dimension at infinity rather than on the topological dimension, see, for instance, \cite{AP14,APV15,AZ24}. Establishing estimate \eqref{est:mainthm} under sharp regularity conditions for higher rank groups would therefore be an interesting direction for further investigation on this point.

It is noteworthy that Ionescu proved in \cite{Ion00} a sharp result analogous to estimate \eqref{est:mainthm} for the wave equation associated with the Laplace--Beltrami operator on noncompact symmetric spaces of rank one. In that case, the time growth on the right-hand side of the estimate is exponential rather than polynomial as in our setting. See also \cite[Remark 3]{MuVa}. Although Damek--Ricci spaces also have exponential volume growth, we still obtain polynomial growth with respect to $|t|$ in estimate \eqref{est:mainthm}. This reflects the difference between the Laplace--Beltrami operator and the distinguished Laplacian.

Theorem \ref{thm:main} is obtained by combining the following improved $L^1$ estimate in terms of the local Hardy space $\fh^1(S)$ (see Subsect. \ref{h1} below), together with the interpolation argument (see Lemma~\ref{lem2.7} below). 

\begin{theorem}\label{thm:main2}
Following the notation of Theorem \ref{thm:MullerVallarino}, there exists $C>0$ such that the solution to the Cauchy problem \eqref{eq: wave} satisfies, for all $t\in\mathbb{R}^*$,
\begin{align*} 
\|u(t,\cdot)\|_{L^1(S,\d\rho)}\leq C\,(1+|t|)\, \big(\|(\Id+\L)^{\frac{\alpha_0}{2}}f\|_{\fh^1(S)}+\|(\Id+\L)^{{\frac{\alpha_1}{2}} }g\|_{\fh^1(S)}\big),
\end{align*}
if and only if 
\begin{align*}
\alpha_0\geq {n-1\over 2}\quad \text{and}\quad \alpha_1\geq {n-1\over 2}-1.
\end{align*}
\end{theorem}

The proof of Theorem \ref{thm:main2} is based on a careful study of the corresponding multipliers of the wave propagators. It combines refined spherical analysis with techniques inspired by Peral \cite{Pe} and Tao \cite{Ta}, and the argument relies crucially on the explicit asymptotic expansions of the spherical function and its derivatives.

We would like to mention that, in the pioneering work \cite{{WaYa}} on the $ax+b$ group, a similar estimate was formulated in terms of the global atomic Hardy space.  This can be viewed as the analogue, in the $ax+b$ setting, of the global Hardy space approach of Miyachi \cite{Mi} and Peral \cite{Pe} on $\R^n$. In light of the interpolation theorem for the local Hardy space, we observe that the endpoint estimate above formulated in terms of the larger local Hardy space is already sufficient for our purposes. After interpolation, however, both endpoint estimates yield the same sharp $L^p$-boundedness result for $1<p<\infty$. It is well known that extending results from the $ax+b$ group to the full class of Damek–Ricci spaces is far from automatic. The nilpotent part of the group is no longer abelian but an $H$-type group endowed with a more intricate algebraic structure. Moreover, such a generalization provides the first insights into the study of the higher rank situations mentioned above.

Before presenting the detailed proofs of the main theorems in Sect. \ref{sec3}, we review in Sect. \ref{sec2} the structure of the Damek–Ricci space $S$ and the analysis on it, where we also recall the definition of the local Hardy space on $S$. Throughout this paper, the notation $A \lesssim B$ between two positive quantities means that there exists a constant $C>0$, independent of all possible variables, such that $A \le CB$. 


\section{Preliminaries}\label{sec2}
In this section, we begin by reviewing the structure of the Damek–Ricci space, followed by a summary of the analysis on it. Most of these materials have already appeared in the literature. See, for instance, \cite{AnDaYa, APV15, CDKR1, CDKR2, D1, D2, DaRi1, DaRi2, MV}. For the reader’s convenience, we recall here some basic facts. In the last part of this section, we review the theory of the local Hardy space on $S$ (see \cite{MV} for further details).

\subsection{Damek-Ricci spaces}\label{subsec2.1}
We first define a Heisenberg-type group, and then its solvable extension, the Damek–Ricci space.
\smallskip

\paragraph{\bf Heisenberg-type groups.}
Let $\mathfrak{n}$ be a real Lie algebra equipped with an inner product $\langle \cdot, \cdot \rangle$ and corresponding norm $|\cdot|$. Let $\fv$ and $\fz$ be complementary orthogonal subspaces of $\fn$ such that $[\fn,\fz ]=\{0\}$ and $[\fv,\fv]\subseteq \fz$. Assume that $\fn$ is of \textit{Heisenberg-type} (or \textit{$H$-type} for short), meaning that for every $z$ in $\fz$, the linear map $J_z:\fv\to \fv$ defined by
$$\langle J_zv,v'\rangle\,=\langle z,[v,v']\rangle\quad\text{for all } v, v'\in \fv$$
is orthogonal. The connected and simply connected Lie group $N$ whose Lie algebra is $\mathfrak{n}$ is called an $H$-type group. Via the exponential map, we identify $N$ with $\mathfrak{n}$ and write
\begin{align*}
\begin{array}{ccc}
\fv\times\fz &\to& N\\
(v,z)&\mapsto &\exp(v+z).
\end{array}
\end{align*}
The multiplication in $N$ is thus given by
\begin{align*}
(v,z)(v',z')
=
\left(v+v',z+z'+\frac12 [v,v']\right)
\quad\forall\,v,\,v'\in \fv,\,\,\,
\forall\,z,\,z'\in\fz.
\end{align*}
Although our definitions and subsequent analysis can, with minor modifications, also be applied to the case where $N$ is abelian, since this case has essentially already been treated in \cite{WaYa}, we shall in this paper focus on the case where $N$ is nonabelian. The group $N$ is a two-step nilpotent group with the unimodular Haar measure  $\d v \d z$. We denote by $Q=m_{\fz}+m_{\fv}/2$ the \textit{homogeneous dimension} of $N$, where $m_{\fz}$ and $m_{\fv}$ are dimensions of $\fz$ and $\fv$, respectively. 
\smallskip

\paragraph{\bf Damek--Ricci spaces.} 
Let $H$ be the unit vector generating the one-dimensional subalgebra $\fa$, and define its adjoint action on $\fn = \fv \oplus \fz$ by
\begin{align*}
\operatorname{ad}(H)v = \frac{1}{2}v \quad \forall\,v \in \fv
\qquad \text{and} \qquad
\operatorname{ad}(H)z = z \quad \forall\,z \in \fz.
\end{align*}
We extend the inner product on $\fn$ to the direct sum $\fs = \fn\oplus\fa$ by requiring $\fa$ to be orthogonal to $\fn$. Then $\fs$ is a solvable Lie algebra, and the corresponding simply connected Lie group $S=N \rtimes A$ is called the \textit{Damek–Ricci space}, where $A= \mathbb{R}_+$ acts on $N$ by the dilation
\begin{align*}
\delta_a(v,z)&=(a^{1\over2}v,az)
\quad\forall\,(v,z)\in N,\,\,\,\forall\,a\in\R_+.
\end{align*}
The product in $S$ is then defined by the rule
\begin{align*}
(v,z,a)(v',z',a')
=
\left(v+a^{1\over2}v',z+az'+\frac12 a^{1\over2}[v,v'],aa'\right)
\qquad\forall\,(v,z,a),\,(v',z',a')\in S.
\end{align*}
We denote by $n=m_{\fv}+m_{\fz}+1$ the dimension of $S$. Recall that the group $S$ is nonunimodular, the right and left Haar measures on $S$ are given, respectively, by
\begin{align*}
\d\rho(v,z,a)=a^{-1}\d v\d z\d a
\qquad\text{and}\qquad
\d\lambda(v,z,a)=a^{-(Q+1)}\d v\d z\d a
\end{align*}
The modular function is thus given by $\delta(v,z,a)=a^{-Q}$. Hereafter, all function spaces on $S$ are defined with respect to the right Haar measure $\d\rho$, unless specified otherwise.

We equip $S$ with the left-invariant Riemannian metric $d$ induced by the inner product
\begin{align*}
\langle (v,z,a),(v',z',a') \rangle
=
\langle v,v' \rangle + \langle z,z' \rangle + aa'
\end{align*}
on $\fs$. It is known that, for all $(v,z,a)\in S$,
\begin{align}\label{e2.1}
\cosh^2\!\left({d((v,z,a),e)\over 2}\right)&=\left({a^{1/2}+a^{-1/2}\over 2}+{a^{-1/2}|v|^2\over 8}\right)^2+{a^{-1}|z|^2\over 4}\\
&={\left(1+a+{|v|^2\over 4}\right)^2+|z|^2\over 4a},
\end{align}
where $e$ denotes the identity of the group $S$, see, for instance, \cite[(2.18)]{AnDaYa}. We shall later use the notation $R(x)$ to denote the distance between the point $x\in S$ and the identity $e$. Let $B_r$ denote the ball in $S$ centered at $e$ with radius $r$. We know from \cite[(1.18)]{AnDaYa} that 
\begin{align}\nonumber
\rho(B_r)\sim\begin{cases} 
r^n&\text{if }0<r\leq 1,\\
\e^{Qr}&\text{if }r>1,
\end{cases}
\end{align}
which implies that the group $S$, equipped with the right Haar measure $\d\rho$, is of exponential growth.

\subsection{Analysis on Damek--Ricci spaces}\label{subsec2.2}
In this subsection, we review the essential analytical tools on the Damek–Ricci space $S$. We first introduce the spherical harmonic analysis on $S$, which serves as a key ingredient in our arguments. Then we recall the polar coordinates and differential operators on $S$. 
\smallskip 

\paragraph{\bf Spherical harmonic analysis.}
A radial function on $S$ is a function that depends only on the distance from the identity. According to \cite[(1.16)]{AnDaYa}, if $\kappa$ is radial, then
\begin{equation}\label{intsin}
\int_S\kappa\,\d \lambda=\int_0^{\infty}\kappa(r)\,A(r)\,\d r,
\end{equation}
where 
\begin{align}\label{def:VolumeA}
A(r)&=2^{m_{ \fv}+m_{\fz}}
\sinh^{m_{ \fv}+m_{\fz}}\! \frac{r}{2}\,
\cosh^{m_{\fz}}\!\frac{r}{2}
\qquad\forall\,r\ge0.
\end{align}
One easily checks that
\begin{align}\label{pesoA}
A(r)&\lesssim \left(\frac{r}{1+r}\right)^{n-1}\e^{Qr}
\qquad\forall\,r\ge0.
\end{align}

For a suitable radial function $\kappa$, one can define the spherical Fourier transform by 
\begin{align}\label{e2.9}
\F(\kappa)(\lambda)
=\nu_n\int_0^\infty \kappa(r)\, \varphi_\lambda(r)\,A(r)\,\d r,
\end{align}
where $\nu_n={2\uppi^{n/2} \Gamma(n/2)^{-1}}$ is the surface area of the unit sphere $\SS^{n-1}$ and $\varphi_{\lambda}$ denotes the \textit{elementary spherical function}, which can be expressed as hypergeometric functions via the formula
\begin{align}\label{e2.10}
\varphi_\lambda(r)={}_2F_1\!\left({Q\over 2}-\i\lambda,{Q\over 2}+\i\lambda;{n\over 2};-\sinh^2{r\over 2}\right),
\end{align}
according to \cite[(3)]{As1}, with ${}_2F_1(a,b;c;z)$ being the classical hypergeometric function given by
\begin{align}\label{ehgf}
{}_2F_1(a,b;c;z)={\Gamma(c)\over \Gamma(b)\,\Gamma(c-b)}\int_0^1t^{b-1}(1-t)^{c-b-1}(1-tz)^{-a}\,\d t.
\end{align}
The spherical function satisfies the estimate
\begin{align}\label{est:spherical}
|\varphi_{\lambda}(r)|
\le \varphi_{0}(r)
\lesssim (1+r)\, \e^{-\frac{Q}{2}r}
\qquad\forall\,\lambda\in\mathbb{C},\,\forall\,r\ge0.
\end{align}
It also satisfies an explicit asymptotic expansion, known as the Harish–Chandra expansion, which plays a crucial role in our argument. A detailed explanation is provided in Sect. \ref{sec3}.

Using the ground spherical function, we obtain the following integration formula, which was already mentioned in \cite[Lemma 2.1]{MuVa}.
\begin{lemma}\label{lem2.2}
For every radial function $\kappa\in C^\infty_c(S)$,
\begin{align*}
\int_S\delta(x)^{1\over2} \kappa(x)\,\d\rho(x)=\int_0^\infty\varphi_0(r)\,\kappa(r)\,A(r)\,\d r=\int_0^\infty \kappa(r)\, V(r)\,\d r,
\end{align*}
where
\begin{align*}
V(r)\lesssim
\begin{cases}
r^{n-1}&\text{if }0<r\leq 1,\\
r\,\e^{{Q\over 2}r}&\text{if }r\geq 1.
\end{cases}
\end{align*}
\end{lemma}

For suitable radial functions $\kappa$ on $S$, we also have the inversion formula of the above Fourier transform:
\begin{align*}
\kappa(r)=
{2^{m_\fz-2}\Gamma({n\over 2})\over \uppi^{{n\over 2}+1}}\int_0^\infty \F (\kappa)(\lambda)\,\varphi_\lambda(r)\,|\bc(\lambda)|^{-2}\,\d\lambda,
\end{align*}
where $\bc$ denotes the Harish-Chandra $\bc$-function given by
\begin{align}\label{e2.11}
\bc(\lambda)={2^{Q-\i\,2\lambda}\Gamma(\i\,2\lambda)\over\Gamma({Q\over2}+\i\lambda)}{\Gamma({n\over2})\over \Gamma({m_\fv\over4}+{1\over2}+\i\lambda)},
\end{align}
see \cite[p. 648]{AnDaYa}. It is well known that the $\bc$-function satisfies
\begin{align}\label{est:cfunction}
\partial_{\lambda}^{k}\bc(\lambda)^{\pm 1}
=
\text{O}(|\lambda|^{\mp \frac{n-1}{2}-k})
\qquad\forall\,k\in\mathbb{N},\,\forall\,|\lambda|\ge1.
\end{align}
In fact, we have the following asymptotic expansion of the $\bc$-function as $|\lambda| \to \infty$:
\begin{align}\label{expansion:cfunction}
\bc(\lambda)
=2^{Q-1} \uppi^{-\frac12} \Gamma\left(\frac{n}{2}\right)
\left((\i\lambda)^{-\frac{n-1}{2}}+O(\lambda^{-\frac{n+1}{2}})\right).
\end{align}
\smallskip

\paragraph{\bf Polar coordinates.}
According to \cite[Sect. 4]{DaRi2}, the group $S$ can be identified with the open unit ball $\BB:=\{(V,Z,b)\in \fv\times \fz\times\R:|V|^2+|Z|^2+b^2<1\}$ in $\fs$ via the map $F:S\to\BB$ defined by
\begin{align*}
F(v,z,a) :={1\over \left(1+a+{|v|^2\over 4}\right)^2+|z|^2}\left(\left(1+a+{|v|^2\over 4}-J_z\right)\!v,2z,-1+\left(a+{|v|^2\over 4}\right)^2+|z|^2\right).
\end{align*}
The inverse map $F^{-1}:\BB\to S$ is given by
\begin{align}\label{e2.3}
F^{-1}(V,Z,b)={1\over (1-b)^2+|Z|^2}\big(2(1-b+J_Z)V,2Z,1-r^2\big),
\end{align}
where $r^2=|V|^2+|Z|^2+b^2$. It follows from \eqref{e2.1} that
\begin{align*}
|F(v,z,a)|^2=1-{4a\over \left(1+a+{|v|^2\over 4}\right)^2+|z|^2}=\tanh^2{R(v,z,a)\over 2},
\end{align*}
where $R(v,z,a)$ denotes the distance between the point $(v,z,a)\in S$ and the identity. The polar coordinates on $S$ are induced by those on $\BB$ via the map $F^{-1}$ and are therefore given by
\begin{align}\label{epol}
x(R,\omega)=F^{-1}\!\left(\left(\tanh{R\over 2}\right) \omega\right),\quad (R,\omega)\in \mathbb{R}_+\times\SS^{n-1}.
\end{align}

Under the polar coordinates, one has the following integration formula:
\begin{align}\label{epolint}
\int_Sf(x)\,\d\lambda(x)=\int_0^\infty\left(\int_{\SS^{n-1}}f(x(r,\omega))\,\d\omega\right)\,A(r)\,\d r,
\end{align}
where $A(r)$ is the volume element given by \eqref{def:VolumeA}. The spherical function given by \eqref{e2.10} is of the form
\begin{align}\label{fml:PhiDelta}
\varphi_\lambda (r)={1\over \nu_n}\int_{\mathbb{S}^{n-1}}\delta(x(r,\omega))^{\i\frac{\lambda}{Q}-\frac12}\,\d\omega.
\end{align}

The next lemma is vital to the proof of our main result.
\begin{lemma}\label{lem2.1}
Let $\ell=0,1$. Then, we have, 
\begin{align*}
\int_{\SS^{n-1}}\left|\partial_R^\ell\!\left[\delta(x(R,\omega))^{-{1\over 2}}\right]\right|\d \omega\lesssim
\begin{cases}
1&\text{if }0<R\leq 1,\\
R\,\e^{-{Q\over 2}R}&\text{if }R>1.
\end{cases}
\end{align*}
\end{lemma}
\begin{proof}
It follows from \eqref{fml:PhiDelta} and \eqref{est:spherical} that
\begin{align}\label{e2.6}
\int_{\SS^{n-1}}\delta(x(R,\omega))^{-{1\over 2}}\,\d\omega=\varphi_0(R)\lesssim\,(1+R)\,\e^{-\frac{Q}{2}R},
\end{align}
which proves the lemma for $\ell=0$. 

Now, let us show that 
\begin{align}\label{e2.7}
\mathbf{I}(R):=\int_{\SS^{n-1}}\left|\partial_R\!\left[\delta(x(R,\omega))^{-{1\over 2}}\right]\right|\d \omega\lesssim\,(1+R)\,\e^{-\frac{Q}{2}R}.
\end{align}
In view of \eqref{e2.3} and \eqref{epol}, if $\omega=(V,Z,b)\in\SS^{n-1}$, then
\begin{align*}
\left|\partial_R\!\left[\delta(x(R,\omega))^{-{1\over 2}}\right]\right|&=\left|\partial_R\!\left[\left({1-\tanh^2{R\over 2}\over \left(1-b\tanh{R\over 2}\right)^2+|Z|^2\tanh^2{R\over 2}}\right)^{Q\over2}\right]\right|\\
&\lesssim \left(\tanh {R\over2}\right)\delta(x(R,\omega))^{-{1\over 2}}\vphantom{\left.+{\cosh^{-Q-2}{R\over 2}\over\left(\left(1-b\tanh{R\over 2}\right)^2+|Z|^2\tanh^2{R\over 2}\right)^{Q/2+1/2}}\right)}+{\cosh^{-Q-2}{R\over 2}\over\left(\left(1-b\tanh{R\over 2}\right)^2+|Z|^2\tanh^2{R\over 2}\right)^{Q+1\over2}}.
\end{align*}
Hence
\begin{align*}
\mathbf{I}(R)&\lesssim \underbrace{\vphantom{\int_{\SS^{n-1}}{\cosh^{-Q-2}{R\over 2}\over\left(\left(1-b\tanh{R\over 2}\right)^2+|Z|^2\tanh^2{R\over 2}\right)^{Q+1\over2}}\,\d\omega}\int_{\SS^{n-1}}\delta(x(R,\omega))^{-{1\over 2}}\,\d\omega}_{\mathbf{I}_1(R)}+\underbrace{\int_{\SS^{n-1}}{\cosh^{-Q-2}{R\over 2}\over\left(\left(1-b\tanh{R\over 2}\right)^2+|Z|^2\tanh^2{R\over 2}\right)^{Q+1\over2}}\,\d\omega}_{\mathbf{I}_2(R)}.
\end{align*}
Clearly by \eqref{e2.6},
\begin{align*}
\mathbf{I}_1(R)\lesssim (1+R)\,\e^{-{Q\over 2}R}.
\end{align*}
Therefore it remains to show that $\mathbf{I}_2(R)$ satisfies the same estimate as above. If $0<R\leq 1$, the integrand inside $\mathbf{I}_2(R)$ is uniformly bounded. Otherwise, by the same change-of-variable argument in the proof of \cite[Theorem 5.12]{DaRi2}, we have
\begin{align*}
\mathbf{I}_2(R)&=\cosh^{-Q-2}{R\over 2}\int_N\left[\left(1+\tanh{R\over 2}+{|v|^2\over 2}+\left(1-\tanh {R\over 2}\right)\!\left({|v|^4\over 16}+|z|^2\right)\right)^2\right.\\
&\qquad\qquad\qquad\qquad\qquad +\left.\vphantom{\left(1+\tanh{R\over 2}+{|v|^2\over 2}+\left(1-\tanh {R\over 2}\right)\!\left({|v|^4\over 16}+|z|^2\right)\right)^2}4|z|^2\tanh^2{R\over 2}\right]^{-{Q+1\over2}}\left(\left(1+{|v|^2\over 4}\right)^2+|z|^2\right)\d v\d z\\
&\lesssim\,\e^{-{Q+2\over 2}R}\int_N\left(1+|v|^2+|z|+\e^{-2R}(|v|^4+|z|^2)\right)^{-(Q+1)}(1+|v|^4+|z|^2)\,\d v\d z.
\end{align*}
Since there exists a constant $C>0$ such that the last integral is bounded by $C\e^R$ for all $R>1$, the proof is therefore completed.
\end{proof}

\paragraph{\bf Distinguished Laplacian and spectral multipliers}
Let $\{X_0,\dots,X_{n-1}\}$ be an orthonormal basis of the Lie algebra $\fs$, where $X_0=H$, the vectors $X_1,\dots,X_{m_{\fv}}$ form an orthonormal basis of $\fv$ and $X_{m_{\fv}+1},\dots,X_{n-1}$ form an orthonormal basis of $\fz$. As usual, we shall identify an element $X\in\fs$ with the corresponding left-invariant differential operator on $S$ given by the Lie derivative $Xf(x)= \left.\frac \d{\d t}f(x\exp tX)\right|_{t=0}.$ Viewing $X_0,X_1,\dots,X_{n-1}$ in this way as left-invariant vector fields, the Riemannian gradient $\nabla$ on $S$ is defined by
\begin{align}\label{riemgrad}
\nabla f=\sum_{i=0}^{n-1}(X_if)\,X_i,
\end{align}
for any function $f$ in $C^\infty(S)$. Its Riemannian norm, denoted by $|\nabla f|_\g$, is given by
\begin{align*}
|\nabla f|_\g=\left(\sum_{i=0}^{n-1}|X_if|^2\right)^{\frac12}.
\end{align*}
The distinguished Laplacian is the left-invariant operator $\L$ defined by 
\begin{align}\label{e1.3}
\L=-\sum_{i=0}^{n-1} X_i^2.
\end{align}
It is essentially self-adjoint on $C^{\infty}_c(S)\subset L^2(S)$. Let $\Delta_S$ be the Laplace--Beltrami operator on $S$, and denote by $\Delta_Q$ the shifted Laplacian $-\Delta_S - Q^2/4$. Recall that the spectra of $\Delta_Q$ on $L^2(S,\d\lambda)$ and $\L$ on $L^2(S,\d\rho)$ are both $[0, \infty)$. In particular, we know from \cite[Proposition 2]{As} that
\begin{align*}
\L f=\delta^{1\over2}\, \Delta_Q(\delta^{-{1\over 2}} f),
\end{align*}
for smooth radial functions $f$ on $S$. It follows from the spectral theorem that, for each bounded Borel measurable function $\psi$ on $\mathbb{R}_+$, both operators $\psi(\sqrt{\Delta_Q})$ and $\psi(\sqrt{\L})$ are $L^2$-bounded, and satisfy 
\begin{align*}
\psi(\sqrt{\L})f=\delta^{1\over2}\psi(\sqrt{\Delta_Q})(\delta^{-{1\over 2}}f),
\end{align*}
for all radial functions $f$ in $C_{c}^{\infty}(S)$.

Let $k_\psi$ and $\kappa_\psi$ be the convolution kernel of $\psi(\sqrt{\L})$ and $\psi(\sqrt{\Delta_Q})$, respectively, that is 
\begin{align*}
\psi(\sqrt{\L})(f)=f*k_\psi\quad\text{and}\quad \psi(\sqrt{\Delta_Q})(f)=f*\kappa_\psi,
\end{align*}
where ``$*$'' denotes the convolution on $S$ defined by
\begin{align}\label{con}
(f*g)(x)=\int_S f(xy)\,g(y^{-1})\,\d\lambda(y),
\end{align}
for all functions $f$ and $g$ in $C_{c}(S)$. According to \cite[Proposition 3.1]{MuVa}, the kernel $\kappa_\psi$ is radial and given by
\begin{align}\label{def:ker kappa}
\kappa_\psi(x)={2^{m_\fz-2}\Gamma({n\over 2})\over \uppi^{{n\over 2}+1}}\int_0^\infty \psi(\lambda)\,\varphi_\lambda(R(x))\,|\bc(\lambda)|^{-2}\,\d \lambda.
\end{align}
Recall that $R(x)$ denotes the distance between the point $x\in S$ and the identity. We deduce from \cite[Proposition 3.1]{MuVa} that 
\begin{align}\label{e3.4}
k_\psi(x)=\delta(x)^{\frac12}\kappa_\psi(x)={2^{m_\fz-2}\Gamma({n\over 2})\over \uppi^{{n\over 2}+1}}\delta(x)^{\frac12}\int_0^\infty \psi(\lambda)\,\varphi_\lambda(R(x))\,|\bc(\lambda)|^{-2}\,\d \lambda.
\end{align}

\smallskip
\subsection{The local Hardy space}\label{h1}
We recall in this subsection the definition of the local atomic Hardy space $\mathfrak{h}^1(S)$, which can be viewed as the analogue, in the setting of Damek--Ricci spaces, of the local Hardy space introduced by Goldberg in the Euclidean case \cite{G}. The local Hardy space was further developed and studied by Meda and Volpi \cite{MV} and by Taylor \cite{T} in more general contexts. It is straightforward to verify that Damek--Ricci spaces satisfy the geometric assumptions of \cite{MV} and \cite{T}, so that the corresponding theory applies to our setting.

A function $a$ in $L^1(S)$ is called a \emph{standard $\fh^1$-atom} if $a$ is  supported in a ball $B$ of radius less than $1$, and satisfies
\begin{enumerate}[\rm (i)]
\item size condition: $\|a\|_{L^2(S)}\leq \rho(B)^{-1/2}$;
\item cancellation condition: $\int a \,\d\rho=0$. 
\end{enumerate}
A \emph{global $\fh^1$-atom} is a function $a$ in $L^1(S)$ supported in a ball $B$ of radius $1$ such that $\|a\|_{L^2(S)}\leq \rho(B)^{-1/2}$. Standard and global $\fh^1$-atoms will be referred to as \emph{admissible atoms}.
The local Hardy space $\fh^1(S)$ is the space of functions $f$ in $L^1(S)$ such that $f=\sum_jc_ja_j$, where $\sum_j|c_j|<\infty$ and $a_j$ are admissible atoms. The norm $\|f\|_{\fh^1(S)}$ is defined as the infimum of $\sum_j|c_j|<\infty$ over all atomic decompositions of $f$.

We now introduce several technical lemmas that will be useful in the subsequent analysis.

\begin{lemma}\label{lem2.7.0}
Let $T$ be a bounded linear operator on $L^2(S)$ with Schwartz kernel $K_T$, which satisfies the following two conditions: 
\begin{asparaenum}[\rm (i)]
\item The local H\"{o}rmonder condition:
\begin{align}\label{eq:lh}
\sup_B\sup_{y,y'\in B}\int_{(2B)^c}|K_T(x,y)-K_T(x,y')|\,\d \rho(x)<\infty,
\end{align}
where the first supremum is taken over all balls $B$ of radii at most $1$;
\item The size condition:
\begin{align}\label{eq:sc}
\sup_{y\in S}\int_{(B(y,2))^c}|K_T(x,y)|\,\d\rho(x)<\infty.
\end{align}
\end{asparaenum}
Then, the operator $T$ is bounded from $\fh^1(S)$ to $L^1(S)$.
\end{lemma}
\begin{proof}
The proof is a combination of \cite[Theorem 8.2]{CaMaMe09} and \cite[Proposition 4.5 (ii)]{CaMaMe10}.
\end{proof}

We conclude this section by providing the boundedness of some operators that will appear in the proof of our main results. Recall that the symbol class is defined by
\begin{align*}
\cS^\sigma=\left\{m\in C^\infty(\R):\|m\|_{\cS^\sigma,k}:=\sup_{\lambda\in\R}\,(1+\lambda^2)^{-{\sigma-k\over 2}}|m^{(k)}(\lambda)|\leq C_k\text{ for all }k\in \mathbb{N}\right\}.
\end{align*}
\begin{lemma}\label{lem2.6}
Suppose that $m\in\cS^\sigma$ is an even symbol. Then the following statements hold.
\begin{asparaenum}[\rm (i)]
\item If $\sigma=0$, then the multiplier $m(\sqrt{\L})$ is bounded from $\fh^1(S)$ to $L^1(S)$, and is bounded on $L^p(S)$ for all $1<p<\infty$.
\item If $\sigma=-1$, then the operator $|\nabla m(\sqrt{\L})\cdot|_\g$, where $\nabla$ denotes the Riemannian gradient on $S$, is bounded from $\fh^1(S)$ to $L^1(S)$, and on $L^p(S)$ for all $1<p<\infty$.
\end{asparaenum}
\end{lemma}

\begin{proof}
The proof follows the same strategy as in the $ax+b$ group treated in \cite{WaYa}. However, since the endpoint estimate there is formulated in terms of a global Hardy space, whereas here we work with the local Hardy space $\fh^1(S)$, we provide the details. 

Let $\{\beta_j\}_{j=0,1,\dots}$ be an inhomogeneous dyadic partition of unity, i.e., there is a $\beta\in C^\infty_c([1/2,2])$ satisfying $\sum_{j=-\infty}^\infty\beta(2^{-j}\cdot)\equiv 1$ on $[0,\infty)$, such that
\begin{align*}
    \beta_0=\sum_{j=-\infty}^0\beta(2^{-j}\cdot),\quad \beta_j=\beta(2^{-j}\cdot)\text{ for }j=1,2,\dots.
\end{align*}

We first show the $\fh^1(S)\to L^1(S)$-boundedness in (i). We assume $m\in \cS^0$ is an even function and denote by $K(\cdot,\cdot)$ the Schwartz kernel of $m(\sqrt{\L})$, and $K_j(\cdot,\cdot)$ the Schwartz kernel of $\beta_j(\sqrt{\L})\,m(\sqrt{\L})$, both with respect to the right Haar measure $\d\rho$.
By Lemma \ref{lem2.7.0}, it is enough to verify the local H\"{o}rmonder condition \eqref{eq:lh} and the size condition \eqref{eq:sc}.
In view of the proof of \cite[Theorem 4.3]{Va2}, for all balls $B$ with radii at most $1$ and $y,y'\in B$, there is a small $\varepsilon>0$ such that
\begin{align}\label{eq:ilh}
    &\int_{(2B)^c}|K(x,y)-K(x,y')|\,\d\rho(x)\\
    \leq &\left(\sum_{j=0}^{\lceil\log(100r_B)\rceil}+\sum_{j=\lceil\log(100r_B)\rceil+1}^\infty\right)\int_{(2B)^c}|K_j(x,y)-K_j(x,y')|\,\d\rho(x)\notag\\
    \leq& C\left(\sum_{j=0}^{\lceil\log(100r_B)\rceil}(2^{-j}r_B)^{-\varepsilon}+\sum_{j=\lceil\log(100r_B)\rceil+1}^\infty2^{-j}r_B\right)=\mathrm O(1),\notag
\end{align}
where $r_B$ is the radius of $B$ and the implicit constant $C$ depends on the finitely many symbol seminorms of $m$. Hence, the local H\"{o}rmonder condition \eqref{eq:lh} for $K(\cdot,\cdot)$ is satisfied. 

It remains to verify the size condition. In view of \cite[Theorem 4.1]{MuVa}, for $x,y\in S$ with $d(x,y)\geq 2$, we have
\begin{align*}
    |K(x,y)|&\leq \sum_{j=0}^\infty |K_j(x,y)|\\
    &\leq C\,\delta(xy)^{1\over 2}\e^{-{Q\over 2}d(x,y)}d(x,y)^{-100n}\sum_{j=0}^\infty2^{-\left(100n-{n-1\over 2}\right)j}\\
    &\lesssim\delta(xy)^{1\over 2}\e^{-{Q\over 2}d(x,y)}d(x,y)^{-100n},
\end{align*}
where the implicit constant $C$ depends on the finitely many symbol seminorms of $m$. By Lemma \ref{lem2.2}, for $y\in S$ we have
\begin{align*}
    \int_{(B(y,2))^c}|K(x,y)|\,\d\rho(x)&\lesssim \int_{(B(y,2))^c}\delta(y^{-1}x)^{-{1\over 2}}\e^{-{Q\over 2}R(y^{-1}x)}R(y^{-1}x)^{-100n}\,\d\lambda(x)\\
    &=\int_S\delta(x)^{1\over 2}\e^{-{Q\over 2}R(x)}R(x)^{-100n}\mathbf{1}_{[2,\infty)}(R(x))\,\d \rho(x)\\
    &\lesssim \int_2^\infty r^{1-100n}\,\d r=\mathrm O(1),
\end{align*}
as desired. Thus \eqref{eq:sc} also holds, and Lemma \ref{lem2.7.0} gives the $\fh^1(S)\to L^1(S)$-boundedness of $m(\sqrt{\L})$.\smallskip

To show the $\fh^1(S)\to L^1(S)$-boundedness in (ii), we assume $m\in \cS^{-1}$ is an even function and denote by $\mathbf K(\cdot,\cdot)$ the Schwartz kernel of $\nabla m(\sqrt{\L})$, and $\mathbf K_{m,j}(\cdot,\cdot)$ the Schwartz kernel of $\nabla \beta_j(\sqrt{\L})\,m(\sqrt{\L})$, both with respect to the right Haar measure $\d\rho$. Note that
\begin{align*}
    \nabla m(\sqrt{\L})=\sum_{j=0}^\infty\nabla \beta_j(\sqrt{\L})\,m(\sqrt{\L})=\sum_{j=0}^\infty\nabla \widetilde m_j(\sqrt{\L})\,\exp(-2^{-2j}\L),
\end{align*}
where $\widetilde m_j(\lambda):=\beta_j(\lambda)\,m(\lambda)\exp(2^{-2j}\lambda^2)$. We additionally denote $\widetilde{\mathbf{K}}_{m,j}(\cdot,\cdot)$ by the Schwartz kernel of $\nabla \widetilde m_j(\sqrt{\L})$. On the one hand, following the proof of \cite[Proposition 5.3]{MuVa},
\begin{align*}
\sup_{j=0,1,\dots}\sup_{y\in S}\int_S|\mathbf{K}_{m,j}(x,y)|_\g\,(1+2^jd(x,y))^\varepsilon\,\d\rho(x)<\infty,
\end{align*}
and
\begin{align}\label{e2.17}
\sup_{j=0,1,\dots}\sup_{y\in S}\int_S|\mathbf{K}_{\widetilde{m},j}(x,y)|_\g\,\d\rho(x)<\infty.
\end{align}
On the other hand, for $j=0,1,\dots$ and $y,y'\in S$, by \cite[Proposition 4.7]{Va2} and \eqref{e2.17}, we have
\begin{align*}
&\int_S|\mathbf{K}_{m,j}(x,y)-\mathbf{K}_{m,j}(x,y')|_\g\,\d\rho(x)\\
&\qquad\leq \int_S\left(\int_S|\mathbf{K}_{\widetilde{m},j}(x,z)|_\g\,\d\rho(x)\right)|H_j(z,y)-H_j(z,y')|\,\d\rho(z)\\
&\qquad \lesssim 2^jd(y,y'),
\end{align*}
where $H_j(\cdot,\cdot)$ is the Schwartz kernel with respect to the right Haar measure $\d\rho$ of the operator $\exp(-2^{-2j}\L)$. Hence arguing analogously to \eqref{eq:ilh}, we derive the local H\"{o}rmonder condition \eqref{eq:lh} for $\mathbf{K}(\cdot,\cdot)$. The proof of the size condition \eqref{eq:sc} for $\mathbf{K}(\cdot,\cdot)$ is the same as the one for $K(\cdot,\cdot)$, using \cite[Theorem 5.1]{MuVa} instead of \cite[Theorem 4.1]{MuVa}. It remains to invoke Lemma \ref{lem2.7.0} to derive the $\fh^1(S)\to L^1(S)$-boundedness part in (ii).

Finally, combined with the $L^2$-boundedness of the operators, the $L^p$-boundedness ($1<p<2)$ in both (i) and (ii) follows from interpolation with the $\fh^1(S)\to L^1(S)$-boundedness via Meda--Volpi's interpolation theorem \cite[Theorem 5(ii)]{MV} and \cite[Theorem 2.6]{Lun18}. The remaining range $p>2$ follows from duality.
\end{proof}

Moreover, the aforementioned interpolation theorem of Meda--Volpi \cite[Theorem 5(ii)]{MV} and \cite[Theorem 2.7]{Lun18} imply the following complex interpolation property for analytic families of operators, which is crucial in the proof of Theorem \ref{est:mainthm}.

\begin{lemma} \label{lem2.7}
	Denote by $\Sigma$ the closed strip $\{\zeta\in\mathbb{C}:\Re \zeta\in[0,1]\}$. Suppose that $\{T_\zeta\}_{\zeta\in\Sigma}$ is a family of uniformly bounded operators on $L^2(S)$ such that the map $\zeta \mapsto \int_S T_\zeta(f)\, g\,\d \rho$ is continuous on $\Sigma$ and analytic in the interior of $\Sigma$, whenever $f,g\in L^2(S)$. Moreover, assume that
	\begin{align*}
		\|T_\zeta f\|_{L^1(S)}\leq A_1\|f\|_{\fh^1(S)}\quad\mbox{for all }f\in L^2(S)\cap \fh^1(S) \mbox{ and }\Re\zeta=0,
	\end{align*}
	and
	\begin{align*}
		\|T_\zeta f\|_{L^2(S)}\leq A_2\|f\|_{L^2(S)}\quad\mbox{for all }f\in L^2(S)\mbox{ and }\Re\zeta=1,
	\end{align*}
	where $A_1$ and $A_2$ are positive constants independent of $\Im \zeta\in\R$. Then for every $\vartheta\in(0,1)$ the operator $T_\vartheta$ is bounded on $L^{q_\vartheta}(S)$ with $1/q_\vartheta=1-\vartheta/2$ and
	\begin{align*}
		\|T_\vartheta f\|_{L^{q_\vartheta}(S)}\leq C\,A_1^{1-\vartheta} A_2^\vartheta\|f\|_{L^{q_\vartheta}(S)}\quad\mbox{for all }f\in L^2(S)\cap L^{q_\vartheta}(S).
	\end{align*}
\end{lemma} 
\medskip


\section{Proof of the main theorems}\label{sec3}
We prove our main theorems in this section, focusing on Theorem \ref{thm:main2}, since Theorem \ref{thm:main} follows from an interpolation argument. The proof of the necessity part of Theorem \ref{thm:main2} is a straightforward adaptation of the argument for the $ax+b$ group given in \cite{WaYa}, and will therefore be omitted. The sufficiency part of Theorem \ref{thm:main2} follows from the results below concerning the corresponding multipliers.

\begin{proposition}\label{prop: main}
Let $m$ be an even symbol. Then, for all $t\in\mathbb{R}^*$, the following results hold.
\begin{asparaenum}[\rm (a)]
\item If $m\in\cS^{-(n-1)/2}$, then
\begin{align*}
\|m(\sqrt{\L})\,\cos(t\sqrt{\L})\|_{\fh^1(S)\to L^1(S)}\lesssim 1+|t|.
\end{align*}

\item If $m\in\cS^{-(n-1)/2+1}$, then
\begin{align*}
\left\|m(\sqrt{\L})\,{\sin(t\sqrt{\L})\over\sqrt{\L}}\right\|_{\fh^1(S)\to L^1(S)}\lesssim 1+|t|. 
\end{align*}
\end{asparaenum}
\end{proposition}

The proof of Proposition \ref{prop: main} relies on the idea, originating from the critical case on Euclidean space, of relating the wave propagator to the spherical measure. See, for instance, \cite{Pe,Ta}. In our setting, the spherical measure is closely related to the spherical function given by \eqref{e2.10}, which therefore requires a detailed analysis. The following expression will be useful later on.

\begin{lemma}
Let $f$ be a suitable function on $S$. Then, for all $t>0$ and $z\in S$, we have
\begin{align}
\varphi_{\!\sqrt{\L}}(t)(f)(z)
={1\over\nu_n}\int_{\SS^{n-1}}\delta(x(t,\omega))^{-{1\over 2}} f(z\cdot x(t,\omega))\,\d \omega,\label{e4.1}
\end{align}
where $\nu_n$ is the surface area of the unit sphere $\SS^{n-1}$, $\delta$ is the modular function on $S$, and $x(t,\omega)$ is the polar coordinate given by \eqref{epol}.
\end{lemma}
\begin{proof}
Let $A(t)$ be the volume element given by \eqref{def:VolumeA}, and $\mu_{t}$ be the Dirac mass on $\R$ centered at $t$. We denote by $\d\sigma_t(r)=\nu_{n}^{-1} A(t)^{-1}\mu_t(r)$ the normalized spherical measure on $S$. Note that $\F(\d\sigma_t)=\varphi_\lambda(t)$. It follows from \eqref{e3.4} that
\begin{align*}
\varphi_{\!\sqrt{\L}}(t)(f)(z)=\big(f*(\delta^{1\over 2}\d\sigma_t)\big)(z).
\end{align*}
Combining with \eqref{con}, \eqref{epolint} and the fact that the modular function $\delta$ defines a group homomorphism from $S$ to $\mathbb{R}_{+}^{*}$, we obtain
\begin{align*}
\varphi_{\!\sqrt{\L}}(t)(f)(z)&={1\over\nu_n}{1\over A(t)}\int_S f(z\cdot x)\,\delta(x^{-1})^{1\over 2}\mu_t(R(x^{-1}))\,\d\lambda(x)\\
&={1\over\nu_n}{1\over A(t)}\int_0^\infty\left(\int_{\SS^{n-1}}\delta(x(r,\omega))^{-{1\over 2}} f(z\cdot x(r,\omega))\,\d \omega\right)\mu_t(r)\, A(r)\,\d r\\
&={1\over\nu_n}\int_{\SS^{n-1}}\delta(x(t,\omega))^{-{1\over 2}} f(z\cdot x(t,\omega))\,\d \omega,
\end{align*}
as expected.
\end{proof}
We now proceed to the proof of Proposition \ref{prop: main}. The argument is divided into two parts, corresponding to the short-time and long-time cases. Both require explicit asymptotic expansions of the spherical function and its derivative. We begin with the large-time analysis, for which the following lemma will be needed.

\begin{lemma}\label{lem:spherical long time}
Let $A(t)$ be the volume element given by \eqref{def:VolumeA}. Suppose that $t > 1$ and $\lambda \ge 1$. Then the following results hold.

\begin{asparaenum}[\rm (a)]
\item There exist symbols $a_t\in\cS^{-(n+1)/2}$, uniformly in $t>1$, such that 
\begin{align*}
\varphi_{\lambda}(t)
=c_n\, A(t)^{-\frac12}\lambda^{-\frac{n-1}{2}}
\cos\!\left(\lambda t - \frac{n-1}{4}\uppi\right)\,
+A(t)^{-\frac12}
\left(\e^{\i \lambda t}a_t(\lambda)+\e^{-\i \lambda t}a_t(-\lambda)\right),
\end{align*}
where $c_n=2^{(n-1)/2}\uppi^{-1/2}\Gamma(n/2)$.

\item There exist symbols $b_t\in\cS^{-(n-1)/2}$, uniformly in $t>1$, such that 
\begin{align*}
\varphi'_{\lambda}(t)
=-c_n\, A(t)^{-\frac12}\lambda^{-\frac{n-3}{2}}
\sin\!\left(\lambda t - \frac{n-1}{4}\uppi\right)\,
+A(t)^{-\frac12}
\left(\e^{\i \lambda t}b_t(\lambda)+\e^{-\i \lambda t}b_t(-\lambda)\right).
\end{align*}
\end{asparaenum}
\end{lemma}

\begin{proof}
This lemma is a consequence of the well-known Harish-Chandra expansion of the spherical function. It follows from \cite[pp. 735--736]{APV15} that
\begin{align}\label{ephiasy}
\varphi_\lambda(t)=2^{-\frac{m_\fz}{2}} A(t)^{-\frac12}
\left(\e^{\i \lambda t}\bc(\lambda)\sum_{\ell=0}^\infty \Gamma_\ell(\lambda)\,\e^{-\ell t}+\e^{-\i \lambda t}\bc(-\lambda)\sum_{\ell=0}^\infty \Gamma_\ell(-\lambda)\,\e^{-\ell t}\right),
\end{align}
where the $\bc$-function is given by \eqref{e2.11}. The coefficients $\Gamma_\ell$ satisfy $\Gamma_0\equiv 1$, and for each $\ell\in\mathbb{N}^{*}$ and $k\in\mathbb{N}$,
\begin{align}\label{egl}
|\partial_\lambda^k\Gamma_\ell(\lambda)|\leq C \ell^d(1+|\lambda|)^{-k-1}
\qquad\forall\,\lambda\in\R^*,
\end{align}
for some constants $C>0$ and $d\geq 1$, see  \cite[Lemma 1]{APV15}. We rewrite \eqref{ephiasy} as
\begin{align}\label{ephiasyr}
\varphi_\lambda(t)
=2^{-\frac{m_\fz}{2}} A(t)^{-\frac12}
\left(\e^{\i \lambda t}\bc(\lambda)+\e^{-\i \lambda t}\bc(-\lambda)\right)+
A(t)^{-\frac12}\left(\e^{\i \lambda t} \widetilde{a}_t(\lambda)+\e^{-\i \lambda t} \widetilde{a}_t(-\lambda)\right),
\end{align}
where
\begin{align*}
\widetilde{a}_t(\lambda)=2^{-\frac{m_\fz}{2}} \bc(\lambda)\sum_{\ell=1}^\infty \Gamma_\ell(\lambda)\,\e^{-\ell t}
\end{align*}
is a symbol in $\cS^{-(n+1)/2}$ uniformly in $t>1$, according to \eqref{est:cfunction} and \eqref{egl}. We then prove (a) by substituting \eqref{expansion:cfunction} into \eqref{ephiasyr}.

For (b), using the formula 
\begin{align*}
{\partial\over \partial z}\big[{}_2F_1(a,b;c;z)\big]={ab\over c}{}_2F_1(a+1,b+1;c+1;z),
\end{align*}
we differentiate both sides of \eqref{e2.10} with respect to $t$ and obtain
\begin{align}\label{e3.100}
\varphi_\lambda'(t)&=-\left[{}_2F_1\!\left({Q\over 2}-\i\lambda,{Q\over 2}+\i\lambda;{n\over 2};\cdot\right)\right]'\!\!\left(-\sinh^2{t\over 2}\right)\,\sinh{t\over 2}\,\cosh{t\over 2}\notag\\ 
&=\left( -{Q^2+4\lambda^2\over 4n}\sinh t \right) {}_2F_1\!\left({Q+2\over 2}-\i\lambda,{Q+2\over 2}+\i\lambda;{n+2\over 2};-\sinh^2{t\over 2}\right).
\end{align}
Another key observation is that, by \eqref{e2.10}, the function ${}_2F_1((Q+2)/2-\i\lambda,(Q+2)/2+\i\lambda;(n+2)/2;-\sinh^2(t/2))$ is the spherical function on the Damek–Ricci space with $\widetilde{m}_\fv=m_\fv$ and $\widetilde{m}_\fz=m_\fz+2$ so that $\widetilde{Q}=Q+2$ and $\widetilde{n}=n+2$. Therefore, we deduce from (a) and \eqref{def:VolumeA} that
\begin{align*}
&{}_2F_1\!\left({Q+2\over 2}-\i\lambda,{Q+2\over 2}+\i\lambda;{n+2\over 2};-\sinh^2{t\over 2}\right)\\
=&\, c_{n+2}\,(\sinh t)^{-1}
A(t)^{-\frac{1}{2}}\lambda^{-\frac{n+1}{2}}
\sin\!\left(\lambda t - \frac{n-1}{4}\uppi\right)
+A(t)^{-\frac12}
\left(\e^{\i\lambda t}\widetilde{a}_t(\lambda)
+\e^{-\i\lambda t}\widetilde{a}_t(-\lambda)\right),
\end{align*}
where $\widetilde{a}_t\in\cS^{-(n+3)/2}$, uniformly in $t>1$. Substituting this equality into \eqref{e3.100}, we obtain
\begin{align*}
\varphi'_\lambda(t)
=-{2^{\frac{n-1}{2}}\Gamma(\frac{n}{2})\over \sqrt{\uppi}}\,
A(t)^{-\frac12} \lambda^{-\frac{n-3}{2}}
\sin\!\left(\lambda t - \frac{n-1}{4}\uppi\right)\,
+A(t)^{-\frac12}
\left(\e^{\i\lambda t}b_t(\lambda)+\e^{-\i\lambda t}b_t(-\lambda)\right),
\end{align*}
for some symbols $b_t\in\cS^{-(n-1)/2}$, uniformly in $t>1$. The proof is therefore completed.
\end{proof}

Let us now turn to the proof of Proposition \ref{prop: main} in the large-time case.
\begin{proof}[Proof of Proposition \ref{prop: main} for $|t|>1$]
Without loss of generality, we may assume that $t>1$. Let $\eta \in C_c^\infty(\R)$ be an even cut-off function such that $\eta \equiv 1$ on $[-1,1]$. To prove part (a), let us assume that $m$ is an even symbol in $\cS^{-(n-1)/2}$. It follows from \cite[Theorem 6.1 (i)]{MuVa} that
\begin{align*}
\|\eta(\sqrt{\L})\,m(\sqrt{\L})\,\cos(t\sqrt{\L})\|_{L^1(S)\to L^1(S)} \leq C\,t,
\end{align*}
for some constant $C>0$ independent of $t$. Therefore, it remains to show that 
\begin{align}\label{e5.8}
\|(1-\eta(\sqrt{\L}))\,m(\sqrt{\L})\, \cos(t\sqrt{\L})\|_{\fh^1(S)\to L^1(S)} \leq C\,t.
\end{align}

To prove \eqref{e5.8}, we combine the elementary identity
\begin{align}\label{ecos}
\cos(\lambda t)=\cos\!\left(\lambda t-{n-1\over 4}\uppi\right)\cos\!\left({n-1\over 4}\uppi\right)-\sin\!\left(\lambda t-{n-1\over 4}\uppi\right)\sin\!\left({n-1\over 4}\uppi\right)
\end{align}
with Lemma \ref{lem:spherical long time} to obtain the following crucial observation: there exist constants $c_1,\,c_2>0$, independent of $t>1$, and symbols $a_{1,t}(\lambda),\,a_{2,t}(\lambda)\in\cS^{-1}$, uniformly in $t>1$, such that for all $\lambda\in\supp(1-\eta)$,
\begin{align*}
\cos(\lambda t)
=c_1\,A(t)^{\frac12}\lambda^{\frac{n-1}{2}}\varphi_\lambda(t)
+c_2\,A(t)^{\frac12}\lambda^{\frac{n-3}{2}}\varphi_\lambda'(t)
+\e^{\i \lambda t}a_{1,t}(\lambda)+\e^{-\i \lambda t}a_{2,t}(\lambda),
\end{align*}
where $A(r)$ is the volume element given by \eqref{def:VolumeA}. Therefore, for such $\lambda$, we have
\begin{align}\label{e5.9}
m(\lambda)\cos( \lambda t)
=c_1\,A(t)^{\frac12} m_1(\lambda)\,\varphi_\lambda(t) 
+ c_2\,A(t)^{\frac12} m_2(\lambda)\,\varphi_\lambda'(t) 
+ \e^{\i \lambda t}b_{1,t}(\lambda)+\e^{-\i \lambda t}b_{2,t}(\lambda)
\end{align}
with $m_1\in\cS^0$, $m_2\in\cS^{-1}$, and $b_{1,t}(\lambda), b_{2,t}(\lambda)\in\cS^{-(n+1)/2}$ uniformly in $t>1$. Note that, on $\supp (1-\eta)$, we can rewrite the remainder terms as 
\begin{align*}
B_t(\lambda)
=\e^{\i \lambda t}b_{1,t}(\lambda)+\e^{-\i \lambda t}b_{2,t}(\lambda)
=\widetilde{b}_{1,t}(\lambda)\cos( \lambda t)+\widetilde{b}_{2,t}(\lambda){\sin( \lambda t)\over \lambda},
\end{align*}
with $\widetilde{b}_{1,t}\in\cS^{-(n+1)/2}$ and $\widetilde{b}_{2,t}\in\cS^{-(n-1)/2}$ uniformly in $t>1$. We then deduce from \cite[Theorem 6.1]{MuVa} that 
\begin{align*}
\|(1-\eta(\sqrt{\L}))\,B_t(\sqrt{\L})\|_{L^1(S)\to L^1(S)}\leq C\, t.
\end{align*}

We now consider the principal terms in \eqref{e5.9}. Let us set
\begin{align*}
M_{1,t}(\lambda)
:=  A(t)^{\frac12}\varphi_\lambda(t)\,\widetilde{m}_1(\lambda)
\qquad\textnormal{and}\qquad 
M_{2,t}(\lambda)
:= A(t)^{\frac12}\varphi_\lambda'(t)\,\widetilde{m}_2(\lambda) ,
\end{align*}
with 
\begin{align*}
\widetilde{m}_1(\lambda):=(1-\eta(\lambda))\, m_1(\lambda)\in\cS^0
\qquad\textnormal{and}\qquad
\widetilde{m}_2(\lambda):=(1-\eta(\lambda))\, m_2(\lambda)\in\cS^{-1}.
\end{align*}
Therefore, the proof of \eqref{e5.8} reduces to showing that
\begin{align}\label{e5.11}
\|M_{j,t}(\sqrt{\L})\|_{\fh^1(S)\to L^1(S)}\leq C\, t, \quad j=1,2.
\end{align}
In the case where $j=1$, we know, on the one hand, from Lemma \ref{lem2.6} (i) that the operator $\widetilde{m}_1 (\sqrt{\L})$ is bounded from $\fh^1(S)$ to $L^1(S)$. On the other hand, combining \eqref{e4.1} with Lemma \ref{lem2.1}, we have
\begin{align*}
\|A(t)^{\frac12}\varphi_{\!\sqrt{\L}}(t)\|_{L^1(S)\to L^1(S)}\leq C\,\e^{\frac{Q}{2}t}\int_{\SS^{n-1}}\delta(x(t,\omega))^{-\frac12}\,\d\omega\leq C\, t,
\end{align*}
from where we obtain
\begin{align*}
\|M_{1,t}(\sqrt{\L})\|_{\fh^1(S)\to L^1(S)}\leq \|A(t)^{\frac12}\varphi_{\!\sqrt{\L}}(t)\|_{L^1(S)\to L^1(S)}\|\widetilde{m}_1(\sqrt{\L})\|_{\fh^1(S)\to L^1(S)}\leq C\, t.
\end{align*}

\medskip
For $j=2$, we apply \eqref{e4.1} to write
\begin{align*}
M_{2,t}(\sqrt{\L})(f)(z)
&=A(t)^{\frac12}\partial_t\!\left[\varphi_{\!\sqrt{\L}}(t)\,\widetilde{m}_2(\sqrt{\L})(f)(z)\right]\\
&={1\over\nu_n}A(t)^{\frac12}\partial_t\!\left[\int_{\SS^{n-1}}\delta(x(t,\omega))^{-\frac12}\widetilde{m}_2(\sqrt{\L})(f)(z\cdot x(t,\omega))\,\d\omega\right].
\end{align*}
We then have $M_{2,t}(\sqrt{\L})(f)(z)=T_{1,t}(f)(z) + T_{2,t}(f)(z)$, with
\begin{align*}
T_{1,t}(f)(z) 
=\frac{1}{\nu_n} A(t)^{\frac12}\int_{\SS^{n-1}}\partial_t[\delta(x(t,\omega))^{-\frac12}]\,\widetilde{m}_2(\sqrt{\L})(f)(z\cdot x(t,\omega))\,\d\omega,
\end{align*}
and
\begin{align*}
T_{2,t}(f)(z)
=\frac{1}{\nu_n} A(t)^{\frac12}\int_{\SS^{n-1}}\delta(x(t,\omega))^{-\frac12}\langle\nabla \widetilde{m}_2(\sqrt{\L})(f)(z\cdot x(t,\omega)),\partial_t[z\cdot x(t,\omega)]\rangle_{\mathbf{g}}\,\d \omega.
\end{align*}
Here $\langle\cdot,\cdot\rangle_{\mathbf{g}}$ denotes the Riemannian metric on $S$ and $\nabla$ is the gradient given by \eqref{riemgrad}.

On the one hand, we deduce from Lemmas \ref{lem2.1} and \ref{lem2.6} (i) that 
\begin{align*}
\|T_{1,t}\|_{\fh^1(S)\to L^1(S)}\leq C\,
\left(\e^{\frac{Q}{2}t}\int_{\SS^{n-1}}|\partial_t[\delta(x(t,\omega))^{-\frac12}]|\,\d\omega\right)\,
\|\widetilde{m}_2(\sqrt{\L})\|_{\fh^1(S)\to L^1(S)}\leq C\,t,
\end{align*}
that is, $T_{1,t}$ is bounded from $\fh^1(S)$ to $L^1(S)$. On the other hand, we know that the factor $|\partial_t [z \cdot x(t,\omega)]|_{\mathbf{g}} \equiv 1$, since $z \cdot x(\cdot,\omega)$ is a geodesic on $S$. Therefore, by successively applying the Cauchy–Schwarz inequality, and Lemmas \ref{lem2.1} and \ref{lem2.6}(ii), we obtain
\begin{align*}
\|T_{2,t}(f)\|_{L^1(S)}\leq C\,
\left(\e^{{Q\over 2}t}\int_{\SS^{n-1}}\delta(x(t,\omega))^{-{1\over 2}}\,\d \omega\right)
\||\nabla\widetilde{m}_2(\sqrt{\L})(f)|_{\mathbf{g}}\|_{L^1(S)}\leq C\,t\,\|f\|_{\fh^1(S)}.
\end{align*}

In conclusion, we have proved that \eqref{e5.11} holds for $j=1,\,2$. The proof of part (a) of Proposition \ref{prop: main} in the long-time case is thus completed. We omit the proof of part (b), as it follows from a similar argument.
\end{proof}

For the short-time case, we need the following lemma, which is similar to Lemma \ref{lem:spherical long time} but involves different remainder terms. The proof relies on an expansion in terms of Bessel functions, rather than on the Harish-Chandra expansion used in the large-time case.

\begin{lemma}\label{lem:spherical short time}
Let $A(t)$ be the volume element given by \eqref{def:VolumeA}. Suppose that $0<t\le1$ and $\lambda > 1$ are such that $\lambda t \ge 1$. Then the following results hold.

\begin{asparaenum}[\rm (a)]
\item There exist symbols $a_t\in\cS^{-(n+1)/2}$, uniformly in $0<t\le1$, such that 
\begin{align*}
\varphi_{\lambda}(t)
=c_n\, A(t)^{-\frac12}\lambda^{-\frac{n-1}{2}}
\cos\!\left(\lambda t - \frac{n-1}{4}\uppi\right)\,
+\e^{\i \lambda t}a_t(\lambda t)+\e^{-\i \lambda t}a_t(-\lambda t),
\end{align*}
where $c_n=2^{(n-1)/2}\uppi^{-1/2}\Gamma(n/2)$.

\item There exist symbols $b_t\in\cS^{-(n-1)/2}$, uniformly in $0<t\le1$, such that
\begin{align*}
\varphi'_{\lambda}(t)
=-c_n\, A(t)^{-\frac12}\lambda^{-\frac{n-3}{2}}
\sin\!\left(\lambda t - \frac{n-1}{4}\uppi\right)\,
+{t^{-1}
\left(\e^{\i \lambda t}b_t(\lambda t)+\e^{-\i \lambda t}b_t(-\lambda t)\right).}
\end{align*}
\end{asparaenum}
\end{lemma}

\begin{proof}
It follows from \cite[Theorem 3.1]{As1} that, for all $0<t\le1$,
\begin{align*}
\varphi_\lambda(t)
=\frac{\Gamma(\frac{n}{2})}{\Gamma(\frac{n-1}{2})}\,
t^{\frac{n-1}{2}} A(t)^{-\frac12}
\sum_{\ell=0}^\infty 2^{\frac{n}{2}-1+\ell}\Gamma\!\left({n-1\over 2}+\ell\right)
c_\ell(t)\,J_{\frac{n}{2}-1+\ell}(\lambda t)\,
(\lambda t)^{-\frac{n}{2}+1-\ell}t^{2\ell},
\end{align*}
where $J_{\mu}$ denotes the classical Bessel function of order $\mu$, and the coefficients ${\red c}_\ell$ satisfy
\begin{align*}
c_0 \equiv 1
\qquad\textnormal{and}\qquad
|c_\ell(t)|\leq C\,4^{-\ell}\quad\forall\,\ell\in\mathbb{N}^*.
\end{align*}
It is well known that $J_{\mu}$ has the asymptotic expansion
\begin{align}\label{e3.1.2}
J_\mu(r)=\left({\uppi\,r\over 2}\right)^{-\frac12}
\cos\!\left(r-\frac{\uppi\mu}{2}-\frac{\uppi}{4}\right)
+\e^{\i r}s_\mu(r)+\e^{-\i r}s_\mu(-r)\quad \text{as }r \to \infty,
\end{align}
where $s_\mu$ is a symbol of order $-3/2$, see, for instance, \cite[p. 356]{St}. Hence, there exists a symbol $a_t \in\cS^{-(n+1)/2}$, uniformly in $0<t\leq 1$, such that
\begin{align*}
\varphi_\lambda(t)
={2^{\frac{n-1}{2}}\Gamma(\frac{n}{2})\over \sqrt{\uppi}}\, 
A(t)^{-\frac12}\lambda^{-\frac{n-1}{2}}
\cos\!\left(\lambda t-\frac{n-1}{4}\uppi\right)
+\e^{\i \lambda t}a_t(\lambda t)+\e^{-\i \lambda t}a_t(-\lambda t).
\end{align*}
Thus, we have proved (a).

Once we prove (a), (b) can be shown in the same spirit as the proof of Lemma \ref{lem:spherical long time}(b). The proof is therefore completed.
\end{proof}

Let us use the above lemma to prove Proposition \ref{prop: main} in the short-time case.

\begin{proof}[Proof of Proposition \ref{prop: main} for $|t|\le1$]
We assume, without loss of generality, that $0<t\le1$. Let us assume that $m$ is an even symbol in $\cS^{-(n-1)/2}$ in order to prove part (a). Recall that $\eta\in C_c^\infty(\R)$ is an even cut-off function satisfying $\eta \equiv 1$ on $[-1,1]$. Define
\begin{align*}
m_{t,0}(\lambda):=\eta( \lambda t)\,m(\lambda)\,\cos( \lambda t),
\end{align*}
which also belongs to $\cS^{-(n-1)/2}$, and hence to $S^{0}$, uniformly in $0<t\le1$. It follows from Lemma \ref{lem2.6} (i) that the operator $m_{t,0}(\sqrt{\L})$ is bounded from $\fh^1(S)$ to $L^1(S)$. The core of this proof therefore lies in establishing
\begin{align}\label{e5.4}
\|(1-\eta(t\sqrt{\L}))\,m(\sqrt{\L})\,\cos(t\sqrt{\L})\|_{\fh^1(S)\to L^1(S)}\leq C
\end{align}
for some constants $C>0$ independent of $t$.

In view of \eqref{ecos} and Lemma \ref{lem:spherical short time}, there exist constants $c_1,\,c_2>0$ such that, on $\supp(1-\eta(t\,\cdot))$,
\begin{align}\label{e5.5}
\cos( \lambda t)
=c_1\, A(t)^{\frac12}\lambda^{\frac{n-1}{2}}\varphi_\lambda(t)
+c_2\,A(t)^{\frac12} \lambda^{\frac{n-3}{2}}\varphi_\lambda'(t)
+\e^{\i \lambda t}a_{1,t}( \lambda t)
+\e^{-\i \lambda t}a_{2,t}( \lambda t)
\end{align}
with $A(t)$ given by \eqref{def:VolumeA} and $a_{1,t}, a_{2,t}\in\cS^{-1}$ uniformly in $0<t\leq 1$. Note that we may write the remainder term as
\begin{align*}
E_{t}(\lambda)
=\e^{\i \lambda t}a_{1,t}( \lambda t)
+\e^{-\i \lambda t}a_{2,t}( \lambda t)
=
\cos( \lambda t)\,\widetilde{a}_{1,t}( \lambda t)
+{\sin( \lambda t)\over \lambda t}\widetilde{a}_{2,t}( \lambda t),
\end{align*}
with $\widetilde{a}_{1,t}\in\cS^{-1}$ and $ \widetilde{a}_{2,t}\in\cS^0$, uniformly in $0<t\leq 1$. 

Note that the remainder term involves only subcritical symbols, so the desired estimate can be obtained by using the method in the proof of \cite[Theorem 6.1]{MuVa}. Following their arguments, there exists a cut-off function $\beta\in C^\infty_c(\R)$, vanishing on $[-1,1]$, such that 
\begin{align*}
(1-\eta( \lambda t))\,m(\lambda)\,\widetilde{a}_{1,t}( \lambda t)=\sum_{\ell=[-\log t]}^\infty M_{\ell,t}(2^{-\ell} \lambda),
\end{align*}
where
\begin{align*}
M_{\ell,t}(\lambda):=\beta(\lambda)\,m(2^\ell\lambda)\,
\widetilde{a}_{1,t}(2^\ell \lambda t),
\end{align*}
satisfies, for all $N=1,2,\dots$, 
\begin{align*}
\|M_{\ell,t}\|_{C^N} \leq C_N\,t^{-1} 2^{-\frac{n+1}{2}\ell},
\end{align*}
uniformly in $\ell\geq [-\log t]$ and $0<t\leq 1$. We deduce from \cite[(6.3)]{MuVa} that 
\begin{align*}
\|M_{\ell,t}(2^{-\ell}\sqrt{\L})\,\cos(t\sqrt{\L})\|_{L^1(S)\to L^1(S)}
\leq C\,t^{-1}2^{-\frac{n+1}{2}\ell} 2^{\frac{n-1}{2}\ell}
=C\,t^{-1}2^{-\ell}.
\end{align*}
It follows that 
\begin{align*}
&\|(1-\eta(t\sqrt{\L}))\,m(\sqrt{\L})\,\widetilde{a}_{1,t}(t\sqrt{\L})\,\cos(t\sqrt{\L})\|_{L^1(S)\to L^1(S)}\\
\leq &\sum_{\ell=[-\log t]}^\infty \|M_{\ell,t}(2^{-\ell}\sqrt{\L})\,\cos(t\sqrt{\L})\|_{L^1(S)\to L^1(S)}\leq C,
\end{align*}
and a similar argument as above yields 
\begin{align*}
\left\|(1-\eta(t\sqrt{\L}))\,m(\sqrt{\L})\,{\widetilde{a}_{2,t}(t\sqrt{\L})}{\sin(t\sqrt{\L})\over t\sqrt{\L}}\right\|_{L^1(S)\to L^1(S)}\leq C.
\end{align*}
The operator $(1 - \eta(t\sqrt{\L}))m(\sqrt{\L})E_t(\sqrt{\L})$, associated with the remainder term, is therefore bounded from $L^1(S)$ to $L^1(S)$.

However, the above method does not allow us to estimate the principal terms in \eqref{e5.5}. Let us define
\begin{align*}
\widetilde{m}_{1,t}(\lambda):=\varphi_\lambda(t)\,m_{1,t}(\lambda)
\quad\textnormal{and}\quad
\widetilde{m}_{2,t}(\lambda):=\varphi'_\lambda(t)\,m_{2,t}(\lambda),
\end{align*}
with
\begin{align*}
m_{1,t}(\lambda):=(1-\eta( \lambda t))\, m(\lambda)\,\lambda^{\frac{n-1}{2}}\in\cS^0\quad \text{and}\quad m_{2,t}(\lambda):=(1-\eta( \lambda t))\, m(\lambda)\,\lambda^{\frac{n-3}{2}}\in\cS^{-1}
\end{align*}
uniformly in $0<t\leq 1$. Given the estimate of the remainder term, the proof of \eqref{e5.4} reduces to showing that both the operators $\widetilde{m}_{1,t}(\sqrt{\L})$ and $\widetilde{m}_{2,t}(\sqrt{\L})$ are bounded from $\fh^1(S)$ to $L^1(S)$. 

On the one hand, since $m_{1,t} \in\cS^0$ uniformly in $0<t\leq 1$, we know from Lemma \ref{lem2.6} (i) that the operator $m_{1,t}(\sqrt\L)$ is bounded from $\fh^1(S)$ to $L^1(S)$. On the other hand, by combining \eqref{e4.1} with Lemma \ref{lem2.1}, we obtain
\begin{align*}
\|\varphi_{\!\sqrt{\L}}(t)\|_{L^1(S)\to L^1(S)}\leq C\int_{\SS^{n-1}}\delta(x(t,\omega))^{-\frac12}\,\d\omega\leq C.
\end{align*}
Therefore, we have 
\begin{align*}
\|\widetilde{m}_{1,t}(\sqrt{\L})\|_{\fh^1(S)\to L^1(S)}\leq \|\varphi_{\!\sqrt{\L}}(t)\|_{L^1(S)\to L^1(S)}\|m_{1,t}(\sqrt{\L})\|_{\fh^1(S)\to L^1(S)}\leq C.
\end{align*}
In other words, the operator $\widetilde{m}_{1,t}(\sqrt{\L})$ is also bounded from $\fh^1(S)$ to $L^1(S)$.

For the operator $\widetilde{m}_{2,t}(\sqrt{\L})$, we apply \eqref{e4.1} once again to write
\begin{align*}
\widetilde{m}_{2,t}(\sqrt{\L})(f)(z)&=\partial_t\!\left[\varphi_{\!\sqrt{\L}}(t)\,m_{2,t}(\sqrt{\L})(f)(z)\right]{-\varphi_{\!\sqrt{\L}}(t)\,(\partial_tm_{2,t})(\sqrt{\L})(f)(z)}\\
&={1\over\nu_n}\partial_t\!\left[\int_{\SS^{n-1}}\delta(x(t,\omega))^{-{1\over 2}}m_{2,t}(\sqrt{\L})(f)(z\cdot x(t,\omega))\,\d\omega\right]\\
&\qquad{-{1\over\nu_n}\int_{\SS^{n-1}}\delta(x(t,\omega))^{-{1\over 2}}(\partial_tm_{2,t})(\sqrt{\L})(f)(z\cdot x(t,\omega))\,\d\omega},
\end{align*}
from where we split up $\widetilde{m}_{2,t}(\sqrt{\L})=T_{1,t}+T_{2,t}$, where the operators $T_j$, with $j=1,2$, are defined by
\begin{align*}
T_{1,t}(f)(z)
=
{1\over\nu_n} \int_{\SS^{n-1}}\partial_t[\delta(x(t,\omega))^{-\frac12}]\,m_{2,t}(\sqrt{\L})(f)(z\cdot x(t,\omega))\,\d\omega
\end{align*}
and
\begin{align*}
T_{2,t}(f)(z)
=
{1\over\nu_n} \int_{\SS^{n-1}}\delta(x(t,\omega))^{-\frac12}\langle\nabla m_{2,t}(\sqrt{\L})(f)(z\cdot x(t,\omega)),\partial_t[z\cdot x(t,\omega)]\rangle_{\mathbf{g}}\,\d \omega.
\end{align*}

On the one hand, we deduce straightforwardly from Lemma \ref{lem2.1} and Lemma \ref{lem2.6} (i) that 
\begin{align*}
\|T_{1,t}\|_{\fh^1(S)\to L^1(S)}
\lesssim\int_{\SS^{n-1}}|\partial_t[\delta(x(t,\omega))^{-\frac12}]|\,\d\omega\
\|m_{2,t}(\sqrt{\L})\|_{\fh^1(S)\to L^1(S)}\leq C.
\end{align*}
On the other hand, by arguments similar to those in the proof of the long-time case, we deduce from Lemmas \ref{lem2.1} and \ref{lem2.6} (ii) that 
\begin{align*}
\|T_{2,t}(f)\|_{L^1(S)}\lesssim \int_{\SS^{n-1}}\delta(x(t,\omega))^{-\frac12}\,\d \omega\cdot\||\nabla{m}_{2,t}(\sqrt{\L})(f)|_{\mathbf{g}}\|_{L^1(S)}\leq C\,\|f\|_{\fh^1(S)}.
\end{align*}
Combining these estimates, we obtain that the operator $\widetilde{m}_{2,t}(\sqrt{\L})$ is bounded from $\fh^{1}(S)$ to $L^1(S)$. The estimate \eqref{e5.4} is therefore proved, and the proof of part (a) of Proposition \ref{prop: main} is complete. The proof of part (b) is analogous to that of part (a) and is therefore omitted.
\end{proof}

Once Proposition \ref{prop: main} is proved, the sufficient parts of Theorems \ref{thm:main} and \ref{thm:main2} can be regarded as its corollaries. We now sketch their proofs below.

\begin{proof}[Proof of Theorem \ref{thm:main2} (sufficiency part)]
According to the spectral theorem, the solution to the Cauchy problem \eqref{eq: wave} is given by
\begin{align*}
u(t,x)=\cos (t\sqrt{\L})(f)(x)+{\sin(t\sqrt{\L})\over \sqrt{\L}}(g)(x).
\end{align*} 
We deduce from Proposition \ref{prop: main} that the estimate
\begin{align*} 
\|u(t,\cdot)\|_{L^1(S)}\lesssim (1+|t|) \big(\|(\Id+\L)^{\frac{\alpha_0}{2}}f\|_{\fh^1(S)}+\|(\Id+\L)^{{\frac{\alpha_1}{2}} }g\|_{\fh^1(S)}\big)
\end{align*}
holds in the critical case where
\begin{align*}
\alpha_0=\frac{n-1}2
\quad\text{and}\quad \alpha_1=\frac{n-1}2-1,
\end{align*}
and therefore also in the supercritical cases.
\end{proof}

\begin{proof}[Proof of Theorem \ref{thm:main} (sufficiency part)]
By using an interpolation argument based on Lemma \ref{lem2.7} (see, for instance, the proof of \cite[Proposition 4.3]{WaYa}), we know that for $1<p<\infty$ and for the even symbols $m_1\in\cS^{-(n-1)|1/p-1/2|}$ and $m_2\in\cS^{-(n-1)|1/p-1/2|-1}$, 
\begin{align*}
\|m_1(\sqrt{\L})\,\cos(t\sqrt{\L})\|_{L^p(S)\to L^p(S)}\lesssim (1+|t|)^{2|1/p-1/2|},
\end{align*}
and
\begin{align*}
\left\|m_2(\sqrt{\L})\,{\sin(t\sqrt{\L})\over\sqrt{\L}}\right\|_{L^p(S)\to L^p(S)}\lesssim 1+|t|,
\end{align*}
which directly implies the sufficient part of Theorem \ref{thm:main}.
\end{proof}

We have proved the sufficient parts of Theorems \ref{thm:main} and \ref{thm:main2}. As mentioned above, the proofs of their necessary parts are straightforward adaptations of the arguments for the $ax+b$ group given in \cite{WaYa}, and are therefore omitted.

\medskip
\noindent\textbf{Acknowledgements.} 
The authors would like to thank the anonymous referee for several valuable comments and suggestions, which helped improve the exposition of the paper.
Y. Wang and L. Yan are supported by National Key R$\&$D Program of China 2022YFA1005700. L. Yan is supported by NNSF
of China 12571111. H.-W. Zhang receives funding from
the Deutsche Forschungsgemeinschaft (DFG) through SFB-TRR 358/1 2023
—Project 491392403, and from Harbin Institute of Technology.


\smallskip
\bibliographystyle{plain}

\end{document}